\documentclass[a4paper,12pt]{amsart}
\usepackage{}
\usepackage{mathrsfs}

\usepackage{amssymb}
\usepackage{latexsym}
\usepackage{amsfonts}
\usepackage{amsmath}
\usepackage{eucal}
\usepackage{bm}
\usepackage{bbm}
\usepackage{graphicx}
\usepackage[english]{varioref}
\usepackage[nice]{nicefrac}
\usepackage[all]{xy}
\usepackage{amsthm}

\newcommand{\Id}{\textrm{id}}

\def\i{^{-1}}
\def\ge{\geqslant}
\def\le{\leqslant}
\def\<{\langle}
\def\>{\rangle}
\def\lto{\hookrightarrow}
\def\||{\parallel}
\def\d{\text{d}}
\def\Dr{\text{D}}
\def\gr{\text{grad}}

\def\sup{\text{sup}}

\def\ud{\underline}

\def\ov{\overline}
\def\u{\underline}

\def\g{\gamma}
\def\G{\Gamma}
\def\d{\delta}

\def\e{\epsilon}

\def\s{\sigma}
\def\t{\tau}
\def\th{\theta}

\def\ZZ{\mathbb Z}
\def\NN{\mathbb N}

\def\RR{\mathbb R}
\def\CC{\mathbb C}

\def\co{\mathcal O}
\def\cp{\mathcal P}

\def\tW{\tilde W}
\def\tw{\tilde w}
\def\tB{\tilde B}

\def\fH{\mathfrak H}

\theoremstyle{plain}
\newtheorem{thm}{Theorem}[section]
\newtheorem*{thm*}{Theorem}
 \newtheorem{prop}[thm]{Proposition}
 \newtheorem{lem}[thm]{Lemma}
 \newtheorem{cor}[thm]{Corollary}

\theoremstyle{definition}

\theoremstyle{remark}

\newtheorem*{rmk}{Remark}
\newtheorem*{claim*}{Claim}

\begin{document}
\author{Xuhua He}
\address{Department of Mathematics, The Hong Kong University of Science and Technology, Clear Water Bay, Kowloon, Hong Kong}
\email{maxhhe@ust.hk}
\author{Sian Nie}
\address{Institute of Mathematics, Chinese Academy of Sciences, Beijing, 100190, China}
\email{niesian@gmail.com}
\title[]{Minimal length elements of finite Coxeter groups}
\keywords{Minimal length element, good element, finite Coxeter group}
\subjclass[2000]{20F55, 20E45}

\begin{abstract}
We give a geometric proof that minimal length elements in a (twisted) conjugacy class of a finite Coxeter group $W$ have remarkable properties with respect to conjugation, taking powers in the associated Braid group and taking centralizer in $W$. 
\end{abstract}

\maketitle

\section*{Introduction}
Let $W$ be a finite Coxeter group. Let $\co$ be a conjugacy class of $W$ and $\co_{\min}$ be the set of elements of minimal length in $\co$. In \cite{GP}, Geck and Pfeiffer showed that the elements in $\co_{\min}$ have remarkable properties with respect to conjugation in $W$ and in the associated Hecke algebra $H$. In \cite{G}, Geck and Michel showed that there exists some element in $\co_{\min}$ that has remarkable properties when taking powers in the associated Braid group. These properties were later generalized to twisted conjugacy classes. See \cite{GKP} and \cite{He1}. In a recent work \cite{L2}, Lusztig showed that the centralizer of a minimal length element in $W$ also has remarkable properties. 

These remarkable properties lead to the definition of determination of ``character tables'' of Hecke algebra $H$. They also play an important role in the study of unipotent representation \cite{L} and in the study of geometric and cohomological properties of Deligne-Lusztig varieties \cite{OR}, \cite{H2}, \cite{BR}, \cite{L2} and \cite{HL}. 

The proofs of these properties were of a case-by-case nature and replied on computer calculation for exceptional types. 

In this paper, we'll give a case-free proof of these properties on $\co_{\min}$ based on a geometric interpretation of conjugacy classes and length function on $W$. Similar ideas will also be applied to affine Weyl groups in a future work. 

In \cite{R}, Rapoport pointed out to us that this paper together with \cite{OR}, \cite{BR} (see also \cite{HL} for a stronger result) gives a computer-free proof of the vanishing theorem \cite[2.1]{OR} on the cohomology of Deligne-Lusztig varieties. This simplifies several steps in Lusztig' classification of representation of finite groups of Lie type \cite{L0}. More precisely, the proof of \cite[Proposition 6.10]{L0} applies without the assumption of $q \ge h$ and many arguments in \cite[Chapter 7-9]{L0} can then be bypassed. 

\section{Geometric interpretation of twisted conjugacy class}

\subsection{} Let $W$ be a finite Coxeter group with generators $s_i$ for $i \in S$ and corresponding Coxeter matrix $M=(m_{i j})_{i, j \in S}$. The elements $s_i$ for $i \in S$ are called {\it simple reflections}. Let $\d: W \to W$ be a group automorphism sending simple reflections to simple reflections. We still denote by $\d$ the induced bijection on $S$. Then $\d(s_i)=s_{\d(i)}$ for all $i \in S$. Set $\tW=\<\d\> \ltimes W$. For any $i \in \ZZ$ and $w \in W$, we set $\ell(\d^i w)=\ell(w)$, where $\ell(w)$ is the length of $w$ in the Coxeter group $W$.

For any subset $J$ of $S$, we denote by $W_J$ the standard parabolic subgroup of $W$ generated by $(s_i)_{i \in J}$ and by $w_J$ the maximal element in $W_J$. We denote by $\tW^J$ (resp. ${}^J \tW$) the set of minimal coset representatives in $\tW/W_J$ (resp. $W_J \backslash \tW$). We simply write ${}^J \tW^J$ for ${}^J \tW \cap \tW^J$. For $\tw \in {}^J \tW^J$, we write $\tw(J)=J$ if the conjugation of $\tw$ sends simple reflections in $J$ to simple reflections in $J$. 

\subsection{}\label{equiv} Two elements $w, w'$ of $W$ are said to be $\d$-twisted conjugate if $w'=\d(x) w x \i$ for some $x \in W$. The relation of $\d$-twisted conjugacy is an equivalence relation and the equivalence classes are said to be $\d$-twisted conjugacy classes.

For $w, w' \in W$ and $i \in S$, we write $w \xrightarrow{s_i}_\d w'$ if $w'=s_{\d(i)} w s_i$ and $\ell(w') \le \ell(w)$.  We write $w \to_\d w'$ if there is a sequence $w=w_0, w_1, \cdots, w_n=w'$ of elements in $W$ such that for any $k$, $w_{k-1} \xrightarrow{s_i}_\d w_k$ for some $i \in S$.

We write $w \approx_\d w'$ if $w \to_\d w'$ and $w' \to_\d w$. It is easy to see that $w \approx_\d w'$ if $w \to_\d w'$ and $\ell(w)=\ell(w')$.

We call $w, w' \in W$ {\it elementarily strongly $\d$-conjugate} if $\ell(w)=\ell(w')$ and there exists $x \in W$ such that $w'=\d(x) w x \i$ and $\ell(\d(x) w)=\ell(x)+\ell(w)$ or $\ell(w x \i)=\ell(x)+\ell(w)$. We call $w, w'$ {\it strongly $\d$-conjugate} if there is a sequence $w=w_0, w_1, \cdots, w_n=w'$ such that for each $i$, $w_{i-1}$ is elementarily strongly $\d$-conjugate to $w_i$. We write $w \sim_\d w'$ if $w$ and $w'$ are strongly $\d$-conjugate.

Now we translate the notations $\to_\d, \sim_\d, \approx_\d$ in $W$ to some notations in $\tW$.

By \cite[Remark 2.1]{GKP}, the map $w \mapsto \d w$ gives a bijection between the $\d$-twisted conjugacy classes of $W$ and the ordinary conjugacy classes of $\tW$ that is contained in $W \d$.

For $\tw, \tw' \in \tW$ and $i \in S$, we write $\tw \xrightarrow{s_i} \tw'$ if $\tw'=s_i \tw s_i$ and $\ell(\tw') \le \ell(\tw)$.  We write $\tw \to \tw'$ if there is a sequence $\tw=\tw_0, \tw_1, \cdots, \tw_n=\tw'$ of elements in $\tW$ such that for any $k$, $\tw_{k-1} \xrightarrow{s_i}\tw_k$ for some $i \in S$. The notations $\approx$ and $\sim$ on $\tW$ are also defined in a similar way as $\approx_\d$, $\sim_\d$ on $W$. Then it is easy to see that for any $w, w' \in W$, $w \to_\d w'$ iff $\d w \to \d w'$, $w \approx_\d w'$ iff $\d w\approx \d w'$ and $w \sim_\d w'$ iff $\d w \sim \d w'$. 

\subsection{} Let $V$ be a real vector space with inner product $(, )$ such that there is an injection $\tW \lto GL(V)$ preserving $(, )$ and for any $i \in S$, $s_i$ acts on $V$ as a reflection. By \cite[Ch.V]{Bour}, such $V$ always exists. Unless otherwise stated, we regard $\tW$ as a reflection subgroup of $GL(V)$. We denote by $|| \cdot ||$ the norm on $V$ defined by $||v||=\sqrt{(v, v)}$ for $v \in V$. 

For any subspace $U$ of $V$, we denote by $U^\bot$ its orthogonal complement. 

For any hyperplane $H$, let $s_H \in GL(V)$ be the reflection along $H$. Let $\fH$ be the set of hyperplanes $H$ of $V$ such that the reflection $s_H$ is in $W$. Let $V^W$ be the set of fixed points by the action of $W$. Since $W$ is generated by $s_H$ for $H \in \fH$, $V^W=\cap_{H \in \fH} H$. 

Even if we start with $V$ with no nonzero fixed points, some pair $(W', V')$ with $(V')^{W'} \neq \{0\}$ appears in the inductive argument in this paper. This is the reason that we consider some vector space other than the one introduced in \cite[Ch.V]{Bour}. 

A connected component $A$ of $V-\cup_{H\in\fH}H$ is called a {\it Weyl chamber}. We denote its closure by $\bar A$. Let $H\in\fH$, if the set of inner points $H_A=(H\cap\bar{A})^\circ \subset H\cap\bar{A}$ spans $H$, then we call $H$ a {\it wall} of $A$ and $H_A$ a {\it face} of $A$.

The Coxeter group $W$ acts simply transitively on the set of Weyl chambers. The chamber containing $C=\{x \in E; (x, e_i)>0 \text{ for all } i \in S\}$ is called the {\it fundamental chamber} which is also denoted by $C$. For any Weyl chamber $A$, we denote by $x_A$ the unique element in $W$ such that $x_A(C)=A$. 

Let $K\subset V$ be a convex subset. A point $x \in K$ is called {\it a regular point} of $K$ if for each $H\in\fH$, $K\subset H$ whenever $x\in H$. The set of regular points of $K$ is open dense in $K$.

\subsection{} Given any element $\tw \in \tW$ and a Weyl chamber $A$, we define $\tw_A=x_A \i \tw x_A$. Then the map $A \mapsto \tw_A$ gives a bijection from the set of Weyl chambers to the conjugacy class of $\tw$ in $\tW$.

For any two chambers $A, A'$, denote by $\fH(A,A')$ the set of hyperplanes in $\fH$ separating $A$ and $A'$. Then $\ell(\tw)=\sharp \fH(C,\tw(C))$ for any $\tw\in \tW$. In general for any Weyl chamber $A$, $$\ell(\tw_A)=\sharp \fH(A, \tw(A)).$$

Let $A, A'$ be Weyl Chambers with a common face $H_A=H_{A'}$, here $H \in \fH$. Then $x_A \i s_H x_A=s_i$ for some $i \in S$. Now $$\tw_{A'}=(s_H x_A) \i \tw (s_H x_A)=s_i x_A \i \tw x_A s_i=s_i \tw_A s_i$$ is obtained from $\tw_A$ by conjugation a simple reflection $s_i$. We may check if $\ell(\tw_{A'})>\ell(\tw_A)$ by the following criterion.

\begin{lem}\label{cr}
We keep the notations as above. Define a map $f_{\tw}: V \to \RR$ by $v \mapsto ||\tw(v)-v||^2$.  Let $h \in H_A$ and $v\in H^\bot$ with $x-\e v\in A$ for sufficiently small $\e>0$. Set $\Dr_vf(h)=\lim_{t\rightarrow0}\frac{f_{\tw}(h+tv)-f_{\tw}(h)}{t}$. If $\ell(\tw_{A'})=\ell(\tw_A)+2$, then $\Dr_vf(h)>0$.
\end{lem}
\begin{proof}
It is easy to see that $\fH(A',\tw A')-\fH(A,\tw A)\subset\{H,\tw H\}$. By our assumption $\sharp\fH(A',\tw A')=\sharp\fH(A,\tw A)+2$. Hence $$\fH(A',\tw A')=\fH(A,\tw A)\sqcup\{H,\tw H\}$$ and $H\neq\tw H$. In particular, $H_A \cap \tw H_A=\emptyset$ and $h\neq\tw(h)$.

Let $L(h,\tw(h))$ be the affine line spanned by $h$ and $\tw(h)$. Then $L(h,\tw(h))-H\cup\tw(H)$ consists of three connected components: $L_{-}=\{h+t(\tw(h)-h); t<0\}$, $L_{0}=\{h+t(\tw(h)-h); 0<t<1\}$ and $L_{+}=\{h+t(\tw(h)-h); t>0\}$. Note that $\fH(A,\tw A)\cap\{H,\tw H\}=\emptyset$, $A\cap L_0$ and $\tw A\cap L_0$ are nonempty. Since $(v,H)=0$ and $h+v,h+(h-\tw(h))$ are in the same component of $V-H$, we have $(v,h-\tw(h))>0$. Similarly we have $(\tw(v),\tw(h)-h)>0$.
Now
\begin{align*}\Dr_vf(h)&=2(\tw(h)-h,\tw(v)-v)\\
                       &=2(\tw(h)-h,\tw(v))+2(h-\tw(h),v)>0.
                       \end{align*}
\end{proof}

\subsection{}
Let $\gr f_{\tw}$ be the gradient of $f_{\tw}$ on $V$, that is, for any vector field $X$ on $V$, $Xf_{\tw}=(X,\gr f_{\tw})$. Here we naturally identify $V$ with the tangent space of any point in $V$. Then it is easy to see that $\gr f_{\tw}(v)=2 (1-{}^t \tw)(1-\tw) v$, where ${}^t\tw$ is the transpose of $\tw$ with respect to $(,)$. Let $C_{\tw}:V\times\RR\rightarrow V$ be the integral curve of $\gr f_{\tw}$ with $C_{\tw}(v,0)=v$ for all $v\in V$. Then $$C_{\tw}(v,t)=\exp(2t(1-{}^t\tw)(1-\tw))v.$$

Let $S(V)=\{v\in V; (v,v)=1\}$ be the unit sphere of $V$. For any $0\neq v \in V$, set $\ov{v}=\frac{v}{||v||}\in S(V)$. Define $p:V-\{0\} \rightarrow S(V)$ by $v\mapsto\lim_{t\rightarrow-\infty}\overline{C_{\tw}(v,t)}$.

In order to study the map $p$, we need to understand the eigenspace of $\tw$ on $V$.

\subsection{}\label{fib} Let $\tw \in \tW$. Let $\G_{\tw}$ be the set of elements $\th \in [0, \pi]$ such that $e^{i \th}$ is an eigenvalue of $\tw$ on $V$.

For $\th \in \G_{\tw}$, we define $$V_{\tw}^\th=\{v \in V;  \tw(v)+\tw^{-1}(v)=2\cos \th v\}.$$ Then $V_{\tw} \otimes_\RR \CC$ is the sum of eigenspaces of $V \otimes_\RR \CC$ with eigenvalues $e^{\pm i \th}$. In particular, if $\th$ is not $0$ or $\pi$, then $V^\th_{\tw}$ is an even-dimensional subspace of $V$ over $\RR$ on which $\tw$ acts as a rotation by $\th$.

Since $\tw$ is a linear isometry of finite order, we have an orthogonal decomposition
$$V=\bigoplus_{\th\in\G_w} V^{\th}_{\tw}.$$

Let $\th_0$ be the minimal element in $\G_{\tw}$ with $V^{\th}_{\tw} \neq V^W$ and $V_{\tw}=V^{\th_0}_{\tw} \cap (V^W)^\bot$.

Now for any $v_\th \in V^\th_{\tw}$, \begin{align*} (1-{}^t \tw)(1-\tw) v_\th & =(1-e^{i \th})(1-e^{-i \th}) v_\th=((1-\cos \th)^2+\sin^2 \th) v_\th \\ &=2(1-\cos \th) v_\th. \end{align*} In particular, let $v \in (V^W)^\bot$, then $v=\sum v_\th$, where $v_\th \in V^\th_{\tw}$ and the summation is over all $\th \in \G_{\tw}$ with $\th \ge \th_0$. Then $C_{\tw}(v, t)=\sum \exp(4 t(1-\cos \th)) v_\th$ and $p(v)=\overline v_{\th_0}$ whenever $v_{\th_0}\neq 0$.

Hence $p((V^W)^\bot-V_{\tw}^\bot)=S(V_{\tw})$ and $p: (V^W)^\bot-V_{\tw}^\bot \to S(V_{\tw})$ is a fiber bundle.

\begin{prop}\label{red}
Let $\tw \in \tW$ and $A$ be a Weyl chamber. Then there exists a Weyl Chamber $A'$ such that $\bar A'$ contains a regular point of $V_{\tw}$ and $\tw_A\rightarrow \tw_{A'}$.
\end{prop}
\begin{proof}
Let $V_{\tw}^{\ge1}\subset V_{\tw}$ be the complement of the set of regular points of $V_{\tw}$. By $\S$\ref{fib}, $p^{-1}(V_{\tw}^{\ge1})$ is a finite union of submanifolds of codimension $\ge 1$. Let $V^{\ge 2}$ be the complement of all chambers and faces in $V$, that is, the skeleton of $V$ of codimension $\ge 2$. Then $C_{\tw}(V^{\ge 2},\RR)\subset V$ is a countable union of images, under smooth maps, of manifolds with dimension at most $\dim V-1$. Let $$D_{\tw}=\{v\in V; v\notin C_{\tw}(V^{\ge 2},\RR)\cup p^{-1}(V_{\tw}^{\ge 1})\cup V_{\tw}^\bot\}.$$ Then $D_{\tw}$ is a dense subset in the sense of Lebesgue measure.

Choose $y\in A\cap D_{\tw}$. Set $x=p(y)\in V_{\tw}$. Then $x$ is a regular point in $V_{\tw}$. There exists $T>0$ such that for any chamber $B$, $x\in \bar B$ whenever $C_{\tw}(y, -T) \in \bar B$.

Now we define $A_i, H_i, h_i, t_i$ as follows.

Set $A_0=A$. Suppose $A_i$ is defined and $A_i \neq A'$, then we set $t_i=\sup\{t<T; C_{\tw}(y, -t) \in \bar{A}_i\}$. Then $t_i \le T$. Set $h_i=C_{\tw}(y, -t_i)$. By the definition of $D_{\tw}$, $h_i$ is contained in a unique face of $A_i$, which we denote by $H_i$. Let $A_{i+1} \neq A_i$ be the unique chamber such that $H_i$ is a common face of $A_i$ and $A_{i+1}$. Then $C_{\tw}(y, -t_i-\e) \in A_{i+1}$ for sufficiently small $\e>0$.

Since the chambers appear in the above list are distinct with each other. Thus the above procedure stops after finitely many steps. We obtain a finite sequence of chambers $A=A_0, A_1, \cdots, A_r=A'$ in this way. Since $C_{\tw}(y,-T)\in \bar A'$, we have $x\in\bar A'$.

Let $v_i \in V$ such that $(v_i, h_i-h)=0$ for $h \in H_i$ and $h_i-\e v_i \in A_i$ for sufficiently small $\e>0$. Since $C_{\tw}(y,-t_i-\e')\in A_{i+1}$ for sufficiently small $\e'>0$, $\Dr_{v_i} f_{\tw}(h_i)=(v_i,(\gr f_{\tw})(h_i))\le0$. Hence by Lemma \ref{cr}, $\ell(\tw_{A_{i+1}})\le \ell(\tw_{A_i})$ and $\tw_{A_i}\rightarrow\tw_{A_{i+1}}$.

Therefore $\tw_A \to \tw_{A'}$ and $\bar A'$ contains a regular point $x$ of $V_{\tw}$.
\end{proof}

\section{Length formula}

\subsection{}\label{aa} The main goal of this section is to give a length formula for the element $\tw_A$ with $\bar A$ containing a regular point of some subspace of $V$ preserved by $\tw$.

Let $K\subset V$ be a convex subset. Let $\fH_K=\{H\in\fH; K\subset H\}$ and $W_K\subset W$ be the subgroup generated by $s_H$ ($H\in\fH_K$). For any two chambers $A$ and $A'$,  set $\fH(A,A')_K=\fH(A,A')\cap\fH_K$.

Let $A$ be a Weyl chamber. We set $W_{K,A}=x_A^{-1}W_K x_A$. If $\bar A$ contains a regular point $k$ of $K$, then we set  $W_{K,A}=W_{I(K, A)}$ is the parabolic subgroup of $W$ generated by simple reflections $I(K, A)=\{s_H\in S; k \in x_A H\}$.

It is easy to see that $\bar A$ contains a regular point of $K$ if and only if it contains a regular point of $K+V^W$. In this case, $I(K, A)=I(K+V^W, A)$. 

If $A'$ is a Weyl chamber such that $\bar A'$ also contains $k$. Then there exists $x \in W_K$ with $x(A)=A'$. We set $x_{A, A'}=x_A \i x x_A$. Then $x_{A, A'} \in W_{K, A}$ and $$\tw_{A'}=(x x_A) \i \tw (x x_A)=(x_A x_{A, A'}) \i \tw (x_A x_{A, A'})=x_{A, A'} \i \tw_A x_{A, A'}.$$ Moreover, $$\ell(x_{A, A'})=\sharp\fH(C, x_{A, A'}(C))=\sharp\fH(x_A(C), x x_A(C))=\sharp\fH(A, A').$$

We first consider the follows special case.

\begin{lem}\label{f1}
Let $\tw\in\tW$ and $K\subset V_{\tw}^\th$ be a subspace such that $\tw(K)=K$. Let $A$ be a Weyl chamber such that $A$ and $\tw(A)$ are in the same connected component of $V-\cup_{H \in \fH_K}H$. Assume furthermore that $\bar A$ contains a nonzero element $v \in K$ such that for each $H \in \fH$, $v, \tw(v) \in H$ implies that $K \subset H$. Then
$$\ell(\tw_A)=\sharp \fH(A, \tw(A))=\frac{\th}{\pi}\sharp(\fH-\fH_K).$$
\end{lem}
\begin{proof}
By our assumption, $\fH(A, \tw(A)) \subset \fH-\fH_K$. Moreover, for any $H \in \fH(A, \tw(A))$, the intersection of $H$ with the closed interval $[v, \tw(v)]$ is nonempty.

If $\th=0$, then $\tw(v)=v$. For $H \in \fH(A, \tw(A))$, $v \in H$ and hence $H \in \fH_K$. That is a contradiction. Hence $\fH(A, \tw(A))=\emptyset$ and $\ell(\tw_A)=\sharp\fH(A, \tw(A))=0$.

If $\th=\pi$, then $\tw(v)=-v$. We see $\fH(A, \tw(A))=\fH- \fH_K$. Thus $\ell(\tw_A)=\sharp(\fH-\fH_K)$.

Now we assume $0<\th<\pi$ and $d$ is the order of $\tw$.  Set $v_i=\tw^i(\bar v) \in S(K)$ for $i \in \ZZ$. Since $\tw$ acts on $K$ by rotation by $\th$, there exists a $2$-dimensional subspace of $K$ that contains $v_i$ for all $i$. Let $S^1$ be the unit circle in this subspace. Let $Q_i \subset S^1$ be the open arc of angle $\th$ connecting $v_i$ with $v_{i+1}$ and $Q'_i=Q_i \sqcup \{v_i\}$.

Let $H \in \fH(\tw^i(A), \tw^{i+1}(A))$. Then by our assumption, $H \in \fH- \fH_K$. If $v_i \notin H$ and $v_{i+1} \notin H$, then $H \cap Q_i \neq \emptyset$. On the other hand, for any $H \in \fH$, if $H \cap Q_i \neq \emptyset$, then $H \in \fH(\tw^i(A), \tw^{i+1}(A))$.

If $H \in \fH -\fH_K$ and $v_i \in H$, then $v_{i-1}, v_{i+1} \notin H$ and $\{v_i\}$ is the intersection of $H$ with the open arc connecting $v_{i-1}$ with $v_{i+1}$ passing through $v_i$. Hence $H$ belongs to exactly one of the two sets: $\fH(\tw^{i-1}(A), \tw^i(A))$ and $\fH(\tw^i(A), \tw^{i+1}(A))$. Therefore

\[\tag{*} \sum_{i=0}^{d-1} \sharp \fH(\tw^i(A), \tw^{i+1}(A))=\sum_{i=0}^{d-1} \sharp \{H \in \fH- \fH_K; H \cap Q'_i \neq \emptyset\}.\]

Notice that each $H \in \fH -\fH_K$ intersects $S^1$ at exactly $2$ points. Hence $H$ appears on the right hand side of $(*)$ exactly $d \th/\pi$-times. Now
$$d \ell(\tw_A)=d \sharp\fH(A, \tw(A))=\frac{d \th}{\pi} \sharp(\fH- \fH_K)$$ and $\ell(\tw_A)=\frac{\th}{\pi} \sharp(\fH- \fH_K)$.
\end{proof}

\begin{prop}\label{f2}
Let $\tw\in\tW$ and $K\subset V_{\tw}^\th$ be a subspace with $\tw(K)=K$. Let $A$ be a Weyl chamber whose closure contains a regular point $v$ of $K$. Then
$$\tw_A=\tw_{K, A} u$$  for some $u \in W_{K, A}$ with $\ell(u)=\sharp \fH(A, \tw(A))_K$ and $\tw_{K, A} \in {}^{I(K, A)} \tW^{I(K, A)}$ with $\tw(I(K, A))=I(K, A)$ and $\ell(\tw_{K, A)}=\frac{\th}{\pi}\sharp(\fH-\fH_K)$. 
\end{prop}

\begin{proof}
We may assume that $A$ is the fundamental Weyl Chamber $C$ by replacing $\tw$ by $\tw_A$. We then simply write $J$ for $I(K, C)$. We have that $\tw=u' \tw' u''$ for some $u', u'' \in W_J$ and $\tw' \in \tW^J$. 

Since $\tw W_J \tw \i=W_J$ and $u', u'' \in W_J$, $\tw' W_J (\tw') \i=W_J$. We also have that $\tw' \in {}^J \tW^J$. Hence $\tw'(J)=J$. Set $u=(\tw') \i u' \tw' u'' \in W_J$. Then $\tw=u' \tw' u''=\tw' u$. 

Since $u$ acts on $K$ trivially, $\tw' K=K$ and $K\subset V_{\tw'}^\th$. By Lemma \ref{f1} $\ell(\tw')=\frac{\th}{\pi}\sharp(\fH-\fH_K)$. Also $$\ell(u)=\sharp\fH(C, u(C))=\sharp\fH(C, u(C))_K=\sharp \fH(C, u\tw'(C))_K=\sharp \fH(C, \tw(C))_K,$$ where the third equality is due to the fact that both $\tw'(C)$ and $C$ belong to $U$.
\end{proof}

\subsection{}\label{cop} 
Let $\tw\in\tW$ and $K\subset V_{\tw}^\th$ be a subspace with $\tw(K)=K$. Let $U$ be a connected component of $V-\cup_{H \in \fH_K} H$. We denote by $\ell(U)$ the number of hyperplanes in $\fH_K$ that separates $U$ and $\tw(U)$. Then by Proposition \ref{f2}, $\ell(\tw_A)=\ell(U)+\frac{\th}{\pi}\sharp(\fH-\fH_K)$ for any Weyl chamber $A \subset U$ such that $\bar A$ contains a regular element of $K$. 

In particular, let $U_0$ be a connected component of $V-\cup_{H \in \fH_{V_{\tw}}} H$ such that $\ell(U_0)$ is minimal among all the connected components. By Proposition \ref{red} and Proposition \ref{f2}, 

(1) $\ell(\tw_A) \ge \ell(U_0)+\frac{\th}{\pi}\sharp(\fH-\fH_{V_{\tw}})$ for any Weyl chamber $A$.

(2) if $A \subset U_0$ and $\bar A$ contains a regular element of $V_{\tw}$, then $\ell(\tw_A)=\ell(U_0)+\frac{\th}{\pi}\sharp(\fH-\fH_{V_{\tw}})$. 

\subsection{} 
Two chambers $A$ and $A'$ are called {\it strongly connected} if they have a common face. For any subspace $K$ of $V$, $A$ and $A'$ are called {\it strongly connected} with respect to $K$ if $\bar A \cap \bar A' \cap K$ spans a codimension $1$ subspace of $K$ of the form $H \cap K$ for some $H \in \fH-\fH_K$. The following result will also be used in the next section. 

\begin{prop}\label{con}
Let $\tw \in \tW$. Let $A$ and $A'$ be Weyl Chambers in the same connected component of $V-\cup_{H \in \fH_K}H$. Assume that $\bar A \cap \bar A' \cap V_{\tw}$ spans $H_0 \cap V_{\tw}$ for $H_0 \in \fH$ and $\tw(H_0 \cap V_{\tw}) \neq H_0 \cap V_{\tw}$, where $H_0$ is the common wall of $A$ and $A'$. Then $$\ell(\tw_A)=\ell(\tw_{A'})=\sharp\fH(A, \tw A)_K+\frac{\th}{\pi}\sharp(\fH-\fH_K).$$
\end{prop}
\begin{proof}
Set $K=V_{\tw}$ and $P=H_0 \cap K$. Then $P$ is a codimension $1$ subspace of $K$. Since $P \neq \tw(P)$, $K=P+\tw(P)$. There exists a regular element $v$ of $P$ such that $v \in \bar A \cap \bar A'$. For $H \in \fH$ with $v, \tw(v) \in H$, $P \subset H$ and $\tw(P) \subset H$ and $K \subset H$. 

Since $\tw(K)=K$, $\tw$ permutes the connected components of $V-\cup_{H \in \fH_K}H$. Let $U$ be the connected component that contains $A$ and $A'$. There exists $u \in W_K$ such that $u \i \tw(U)=U$. By Lemma \ref{f1}, $\sharp\fH(A,  u \i \tw(A))=\frac{\th}{\pi}\sharp(\fH-\fH_K)$. 

Now we define a map $\phi:\fH(A, \tw A)-\fH(A, \tw A)_K\rightarrow \fH$ by \[\phi(H)=\begin{cases}u^{-1}(H), &\text{ if $\tw(v)\in H$; }\\ H, &\text{ otherwise. } \end{cases}\]

Notice that $u \tw(v)=\tw(v)$. Thus $\tw(v) \in H$ if and only if $\tw(v) \in u \i(H)$. Therefore the map $\phi$ is injective. Let $H\in\fH(A, \tw A)-\fH(A, \tw A)_K$. If $\tw(v) \in H$, then $v \notin H$. Hence $H$ separates $v$ from $\tw A$, hence $\phi(H)=u^{-1}H$ separates $u^{-1}(v)=v$ from $u^{-1}\tw A$ and $\phi(H)\in\fH(A, u^{-1}\tw A)$. If $\tw(v) \notin H$, then $\phi(H)=H$ separates $u \i \tw(v)=\tw(v)$ from $A$ and hence $\phi(H) \in \fH(A,  u \i \tw A)$. Thus the image of $\phi$ is contained in $\fH(A,  h \i \tw(A))$. 

On the other hand, let $H \in \fH(A,  u \i \tw(A))$. Since $A$ and $u \i \tw(A)$ are both in $U$, $H \notin \fH_K$. If $\tw(v) \in H$, then $H$ separates $v$ from $u \i\tw(A)$ and $u(H)$ separates $v$ from $\tw (A)$. Hence $u(H) \in \fH(A,  \tw(A))$. If $\tw(v) \notin H$, then $H$ separates $\tw(v)$ from $A$ and hence $H \in \fH(A,  \tw(A))$. 

Therefore the image of $\phi$ is $\fH(A,  \tw(A))$. Since $\phi$ is bijective, \begin{align*} \ell(\tw_A) &=\sharp\fH(A,  \tw(A))=\sharp\fH(A,  \tw(A))_K+\sharp\fH(A, u \i \tw(A)) \\ &=\sharp\fH(A,  \tw(A))_K+\frac{\th}{\pi}\sharp(\fH-\fH_K).\end{align*}

Similarly, $\ell(\tw_{A'})=\sharp \fH(A', \tw(A'))_K+\frac{\th}{\pi}\sharp(\fH-\fH_K)$. Since $A$ and $A'$ are in the same connected component of $V-\cup_{H \in \fH_K} H$, $\fH(A,  \tw(A))_K=\fH(A', \tw(A'))_K$. The Proposition is proved.
\end{proof}

\section{Strongly conjugacy}

The following result is proved in \cite{GP}, \cite{GKP} via a case-by-case analysis.

\begin{thm} Let $(W, S)$ be a finite Coxeter group and $\d: W \rightarrow
W$ be an automorphism sending simple reflections to simple reflections. Let $\co$ be a $\d$-twisted
conjugacy class in $W$ and $\co_{\min}$ be the set of minimal length
elements in $\co$. Then

(1) For each $w \in \co$, there exists $w' \in \co_{\min}$ such that
$w \rightarrow_{\d} w'$.

(2) Let $w, w' \in \co_{\min}$, then $w \sim_{\d} w'$.
\end{thm}

By $\S$\ref{equiv}, we may reformulate it as follows. 

\begin{thm}\label{main}
Let $(W, S)$ be a finite Coxeter group and $\d: W \rightarrow
W$ be an automorphism sending simple reflections to simple reflections. Set $\tW=\<\d\> \ltimes W$. Let $\co$ be a $W$-conjugacy class in $\tW$ and $\co_{\min}$ be the set of minimal length
elements in $\co$. Then

(1) For each $w \in \co$, there exists $w' \in \co_{\min}$ such that
$w \rightarrow w'$.

(2) Let $w, w' \in \co_{\min}$, then $w \sim w'$.
\end{thm}

The main purpose of this section is to give a case-free proof of this result. 

\subsection{}\label{par} We first discuss some relation between a conjugacy class in $\tW$ and in a ``smaller'' subgroup. This is a special case of ``partial conjugation'' method in \cite{He1}. 

Let $J \subset S$. Let $\tw \in {}^J \tW^J$ be an element with $\tw(J)=J$. We denote by $\d'$ the automorphism on $W_J$ defined by the conjugation of $\tw$. Set $\tW'=\<\d'\> \ltimes W_J$. Let $\ell'$ be the length function on $\tW'$. Then the map $$f: \tW' \to \tW, \qquad \d'x \mapsto \tw x$$ is equivariant for the conjugation action of $W_J$ and $\ell(f(\d' x))=\ell(x)+\ell(\tw)=\ell_1(\d' x)+\ell(\tw)$. Hence for any $x, x' \in W_J$, $\tw x \to \tw x'$ if and only if $\d' x\to \d' x'$ (in $\tW'$). Similar results hold for $\sim$ and $\approx$. 

\subsection{}\label{p1} We prove Theorem \ref{main} (1). We argue by induction on $\sharp W$. The statement holds if $W$ is trivial. Now we assume that the statement holds for any $(W', S', \d')$ with $\sharp W'<\sharp W$. 

Any element in the conjugacy class of $\tw$ is of the form $\tw_{A'}$ for some Weyl chamber $A'$. Set $K=V_{\tw}$. By Proposition \ref{red}, $\tw_{A'} \to \tw_A$ for some Weyl chamber $A$ such that $\bar A$ contains a regular element of $K$. 

Set $J=I(K, A)$. By Proposition \ref{f2}, $\tw_A=\tw_{K, A} u$, where $u \in W_J$, $\tw_{K, A} \in {}^J \tW^J$ with $\tw_{K, A}(J)=J$ and $\ell(\tw_{K, A})=\frac{\th}{\pi}\sharp(\fH- \fH_K)$. 

Let $\d_1$ be the automorphism on $W_J$ defined by the conjugation of $\tw_{K, A}$. Set $\tW_1=\<\d_1\> \ltimes W_J$. Since $K$ is not contained in $V^W$, there exists $H \in \fH$ such that $K \nsubseteq H$. Thus $W_J \nsubseteq W$. Now by induction hypothesis on $\tW_1$, there exists $u' \in W_J$ such that $\d_1 u'$ is a minimal length element in its conjugacy class in $\tW_1$ and $\d_1 u\to \d_1 u'$. Then $\tw_{K, A} u\to \tw_{K, A} u'$. 

Let $U$ be the connected component of $V-\cup_{H \in \fH_K} H$ that contains $A$. Let $x \in W_J$ with $\d_1 u'=x \i \d_1 u x$. Set $B=x_A x x_A \i$ and $U'=x_A x x_A \i(U)$. Then $\tw_B=x \i \tw_A x=\tw_{K, A} u'$. 
Since $\d_1 u'$ is a minimal length element in its conjugacy class in $\tW_1$, $\ell(U')$ is minimal among all the connected component of $V-\cup_{H \in \fH_K} H$. Hence by Proposition \ref{f2} and $\S$\ref{cop}, $\tw_B=\tw_{K, A} u'$ is a minimal length element in the conjugacy class of $\tw$. Part (1) of Theorem \ref{main} is proved. 

\

To prove Theorem \ref{main} (2), we need the following result. 

\begin{lem}\label{gen}
Assume that Part (2) of Theorem \ref{main} holds for $(W', S', \d')$ with $\sharp W'<\sharp W$. Let $\tw \in \tW$ and $K\subset V_{\tw}$ be a nonzero subspace with $\tw(K)=K$. Let $A,A'$ be two chambers whose closures contain a common regular point $x$ of $K$. Assume further that $\tw_A$ and $\tw_{A'}$ are of minimal length in their conjugacy class of $\tW$. Then $\tw_A\sim\tw_{A'}$.
\end{lem}
\begin{proof}
Set $J=I(K, A)$. By Proposition \ref{f2}, $\tw_A=\tw_{K,A} u$, where $u\in W_J$, $\tw_{K, A} \in {}^J \tW^J$ with $\tw_{K, A}(J)=J$. We define $\d_1, \tW_1$ as in $\S$\ref{p1}. Let $\ell_1$ be the length function on $\tW_1$. Let $x \in W_K$ with $x(A)=A'$. Set $y=x_A \i x x_A$. Then $y \in W_J$ and $\tw_{A'}=y \i \tw_{K,A} u y$. Since $\ell(\tw_{A'})=\ell(\tw_A)$, $\ell_1(y \i \d_1 u y)=\ell_1(\d_1 u)$. Hence by induction hypothesis on $\tW_1$, $y \i \d_1 u y \sim \d_1 u$. By $\S$\ref{par}, $\tw_A' \sim \tw_A$. 
\end{proof}

\subsection{}\label{p2} Now we prove Theorem \ref{main} (2). As in $\S$\ref{p1}, we assume that the statement holds for any $(W', S', \d')$ with $\sharp W'<\sharp W$. Set $K=V_{\tw}$. Let $\tw_A, \tw_{A'} \in \co_{\min}$. By Proposition \ref{red}, it suffices to consider the case where $\bar A$ and $\bar A'$ both contain regular elements of $K$. 
Let $U$ (resp. $U'$) be the connected component of $V-\cup_{H \in \fH_K} H$ that contains $A$ (resp. $A'$). Then by $\S$\ref{cop}, $\ell(U)=\ell(U')$ are minimal among all the connected component of $V-\cup_{H \in \fH_K} H$. 

We define $\d_1, \tW_1$ as in $\S$ \ref{p1}. Let $x \in W_K$ with $x(U)=U'$. Set $y=x_{A, x(A)}$. Then $\tw_{x(A)}=y \i \tw_A y$ and $$\ell(\tw_{x(A)})=\ell(U')+\frac{\th}{\pi}\sharp(\fH- \fH_K)=\ell(U)+\frac{\th}{\pi}\sharp(\fH- \fH_K)=\ell(\tw_A).$$ Hence $y \i \d_1 u y$ is a minimal length element in the conjugacy class of $\tW_1$ that contains $\d_1 u$. By induction hypothesis on $\tW_1$, $y \i \d_1 u y \sim \d_1 u$. Hence $\tw_{x(A)}=y \i \tw_{K, A} u y \sim \tw_{K, A} u=\tw_A$. 

Thus to prove part (2), it suffices to prove that $\bar A$ and $\bar A'$ are in the same connected component of $V-\cup_{H \in \fH_K} H$ and $\tw_A, \tw_{A'} \in \co_{\min}$, then $\tw_A \sim \tw_{A'}$. 

There exists a sequence of chambers $A=A_0,\cdots,A_r=A'$ in the same connected component of $V-\cup_{H \in \fH_K} H$ whose closures contain regular points of $K$ and for any $i$, $A_i$ and $A_{i+1}$ are strongly connected to $A_{i+1}$ with respect to $K$. By $\S$\ref{cop}, $\tw_{A_i}$ is of minimal length. It suffices to prove that $\tw_{A_i} \sim \tw_{A_{i+1}}$ for any $i$. 

By definition, $\bar{A_i}\cap\bar A_{i+1}\cap K$ spans $H_0\cap K$ for some $H_0\in\fH-\fH_K$. Set $P=H_0 \cap K$. Then $\bar A_i$ and $\bar A_{i+1}$ contains a common regular element $v$ of $P$. 

Case 1: $\tw(P) \neq P$. There is a sequence of chambers $A_i=B_0,\cdots,B_t=A_{i+1}$ in the same component of $V-\cup_{H\in\fH_K}H$ such that for any $j$, $v \in \bar B_j$, $B_j$ and $B_{j+1}$ share a common wall. By Proposition \ref{con}, $\ell(\tw_{B_0})=\ell(\tw_{B_1})=\cdots=\ell(\tw_{B_t})$. Since $B_j$ and $B_{j+1}$ are strongly connected, $\tw_{B_j}\approx\tw_{B_{j+1}}$. In particular, $\tw_{A_i} \approx\tw_{A_{i+1}}$. 

Case 2: $P=\tw(P)$ and $\dim(K) \ge 2$. Then $\dim(P) \ge 1$ is a nonzero subspace of $V_{\tw}$. Apply Lemma \ref{gen} for $P$, we obtain that $\tw_A\sim\tw_{A'}$.

Case 3: $\dim(K)=1$. Then $P=\{0\}$. By $\S$\ref{fib}, $\th_0=0$ or $\pi$. If $\th_0=\pi$, then $\tw$ acts as $-\Id$ on $(V^W)^\bot$, hence $\tw_{A_i}=\tw_{A_{i+1}}$ acts as $-\Id$ on $(V^W)^\bot$. Now assume that  $\th_0=0$. Let $v$ be a regular element of $K$ with $v \in \bar A_i$. Then $-v \in \bar A_{i+1}$. Since $\tw(v)=v$, then $\fH(A_i,\tw(A_i))=\fH(A_{i+1}, \tw(A_{i+1}))\subset\fH_K$. So \begin{align*} \fH(A_i,\tw(A_{i+1}))-\fH(A_i,\tw(A_{i+1}))_K &=\fH(A_i, A_{i+1})-\fH(A_i,A_{i+1})_K \\ &=\fH(A_i,A_{i+1}).\end{align*} Let $x\in W$ with $x(A_{i+1})=A_i$. By $\S$\ref{aa}, $\tw_A x_{A_i, A_{i+1}}=x_{A_i} \i \tw x x_{A_i}$ and
\begin{align*}
\ell(\tw_{A_i} x_{A_i, A_{i+1}})&=\sharp \fH(C,\tw_{A_i} x_{A_i, A_{i+1}}(C))=\sharp \fH(x_{A_i}(C),\tw x x_{A_i}(C)) \\ &=\sharp(A_i, \tw(A_{i+1}))=\sharp\fH(A_i,\tw(A_{i+1}))_K+\sharp\fH(A_i,A_{i+1}) \\  &=\sharp\fH(A_i,\tw(A_i))_K+\sharp\fH(A_i,A_{i+1}) \\ &=\ell(\tw_{A_i})+\ell(x_{A_i, A_{i+1}}).
\end{align*}
Hence $\tw_{A_i} \sim \tw_{A_{i+1}}$. 

\section{Elliptic conjugacy class}

\subsection{} 
We call a conjugacy class $\co$ of $\tW$ (or an element of it) {\it elliptic} if for some (or equivalently, any) element $\tw \in \co$, points in $V$ fixed by $\tw$ are contained in $V^W$. By \cite[Lemma 7.2]{He1}, $\co$ is elliptic if and only if $\co \cap (\<\d\>\ltimes W_J)=\emptyset$ for any proper subset $J$ of $S$ with $\d(J)=J$. In particular, the definition of elliptic conjugacy class/element is independent of the choice of $V$. 

We've shown in the previous section that any two minimal length element in a conjugacy class of $\tW$ are strongly conjugate. In this section, we'll obtained a stronger result for elliptic conjugacy classes. 

Let $\tw \in \tW$ be a minimal length element in its conjugacy class. Let $\cp_{\tw}$ be the set of sequences $\mathbf i=(i_1, \cdots, i_t)$ of $S$ such that $$\tw \xrightarrow{i_1} s_{i_1} \tw s_{i_1} \xrightarrow{i_2} \cdots \xrightarrow{i_t} s_{i_t} \cdots s_{i_1} \tw s_{i_1} \cdots s_{i_t}.$$ Since $\tw$ is a minimal element element, all the elements above are of the same length. We call such $\mathbf i$ a {\it path} from $\tw$ to $s_{i_t} \cdots s_{i_1} \tw s_{i_1} \cdots s_{i_t}$. Let $\cp_{\tw, \tw}$ be the subset of $\cp$ consisting of all paths from $\tw$ to itself. 

Let $W_{\tw}=\{w \in W; \ell(w \i \tw w)=\ell(\tw)\}$ and $Z_W(\tw)=\{w \in W; w \tw=\tw w\} \subset W_{\tw}$. Then we have a natural map $$\t_{\tw}: \cp_{\tw} \to W_{\tw}, \qquad (i_1, \cdots, i_t) \mapsto s_{i_1} \cdots s_{i_t}.$$

Let $C_{\tw}$ be the set of all Weyl chambers $A$ with $\ell(\tw_A)=\ell(\tw)$. Then the map $A \mapsto x_A$ gives a bijection between $C_{\tw}$ and $W_{\tw}$. 

We call an element $v \in V$ {\it subregular} if it is either regular in $V$ or regular in $H$ for some $H \in \fH$. Let $V^{subreg} \subset V$ be the set of all subregular element. Then $V-V^{subreg}$ is a finite union of codimension $2$ subspaces. 

\begin{lem}\label{conn}
Let $A$, $A'$ be Weyl chambers in $C_{\tw}$. Then there is a path from $\tw_A$ to $\tw_{A'}$ if and only if $A$ and $A'$ are in the same connected component of $(\cup_{A \in C_{\tw}} \bar A) \cap V^{subreg}$. 
\end{lem}

\begin{proof} If $A$ and $A'$ are in the same connected component of $(\cup_{A \in C_{\tw}} \bar A) \cap V^{subreg}$, then there is a  sequence of Weyl chambers $A=A_0, A_1, \cdots, A_r=A'$ in $C_{\tw}$ such that $A_i$ and $A_{i+1}$ are strongly connected. Let $H_i$ be the common wall of $A_i$ and $A_{i+1}$. Then $x_{A_i} \i s_{H_i} x_{A_i}=s_i$ for some $i \in S$ and $$\tw_A \xrightarrow{i_0} \tw_{A_1} \xrightarrow{i_1} \cdots \xrightarrow{i_{r-1}} \tw_{A'}.$$ Therefore $(i_0, \cdots, i_{r-1}) \in \cp_{\tw}$. 

On the other hand, any path $(i_0, \cdots, i_{r-1})$ from $\tw_A$ to $\tw_{A'}$ gives a sequence $A=A_0, A_1, \cdots, A_r=A'$ in $C_{\tw}$ such that $A_i$ and $A_{i+1}$ are strongly connected. Hence $A$ and $A'$ are in the same connected component of $(\cup_{A \in C_{\tw}} \bar A) \cap V^{subreg}$. 
\end{proof}

\

Our main result in this section is 

\begin{thm}\label{main2}
Let $\co$ be an elliptic conjugacy class of $\tW$ and $\tw \in \co_{\min}$. Then the map $\t_{\tw}: \cp_{\tw} \to W_{\tw}$ is surjective. 
\end{thm}

\begin{proof} 
We argue by induction on $\sharp W$. The statement holds if $W$ is trivial. Now assume that the statement holds for $(W', S', \d')$ with $\sharp W'<\sharp W$. 

Set $K=V_{\tw}$ and $Z=(\cup_{A \in C_{\tw}} \bar A) \cap V^{subreg}$. By Lemma \ref{conn}, it suffice to show that $Z$ is connected. 

Let $A \in C_{\tw}$. Then by the proof of Proposition \ref{red}, there exists a Weyl Chamber $A'$ such that $\bar A'$ contains a regular element of $K$ and there is a curve in $Z$ connecting $A$ and $A'$. Now it suffices to show that for any $A, B \in C_{\tw}$ such that $\bar A$ and $\bar B$ contain regular element of $K$, $A$ and $B$ are in the same connected component of $Z$. 

Let $U$ be the connected component of $V-\cup_{H \in \fH_K} H$ that contains $A$. Let $x \in W_K$ with $x(B) \subset U$. Let $J=I(K, B)$ and $\d_1$ be the automorphism on $W_J$ defiend by the conjugation of $\tw_{K, B}$. Set $\tW_1=\<\d_1\> \ltimes W_J$ and $V_1=\sum_{H \in \fH, s_H \in W_J} H^\bot$. The action of $W$ on $V$ induces an injection $W_J \to GL(V_1)$. Also $\d_1(V_1)=V_1$. Hence we may regard $\tW_1$ as a reflection subgroup of $V_1$. By Proposition \ref{f2}, $\tw_B=\tw_{K, B} u$ for some $u \in W_J$. Let $v \in V_1$ with $\tw_B(v)=v$. Then $v \in V_1\cap V^W \subset V_1^{W_J}$. Thus $\d_1 u$ is elliptic in $\tW_1$. 

By $\S$\ref{cop}, $\ell(\tw_{x(B)})=\ell(\tw_B)$. Hence by $\S$\ref{par}, $\d_1 u$ and $x_{B, x(B)} \i \d_1 u x_{B, x(B)}$ are both of minimal length in their conjugacy class in $\tW_1$. Thus by induction hypothesis on $\tW_1$, $\d_1 u \approx x_{B, x(B)} \i \d_1 u x_{B, x(B)}$. Hence by $\S$\ref{par}, $\tw_B \approx \tw_{x(B)}$. Hence by Lemma \ref{conn}, $B$ and $x(B)$ are in the same connected component of $Z$. 

Now $A$ and $x(B)$ are in the same connected component $U$ of $V-\cup_{H \in \fH_K} H$. By $\S$\ref{p2}, there exists a sequence of chambers $A=A_0,\cdots,A_r=x(B)$ in $C_{\tw}$ such that for any $i$, $A_i$ and $A_{i+1}$ are strongly connected with respect to $K$. By definition, $\bar{A_i}\cap\bar A_{i+1}\cap K$ spans $H_0\cap K$ for some $H_0\in\fH-\fH_K$. If $\th_0=\pi$, then $\tw_{A_i}=\tw_{A_{i+1}}$ acts as $-\Id$ on $(V^W)^\bot$. If $\th_0 \neq \pi$, then any $\tw$-stable subspace of $K$ is even-dimensional and $\tw(H_0 \cap K) \neq H_0 \cap K$. Thus we are in case 1 of $\S$\ref{p2}. Hence $\tw_{A_i} \approx \tw_{A_{i+1}}$. Therefore $\tw_A \approx \tw_{x(B)}$. By Lemma \ref{conn}, $A$ and $x(B)$ are in the same connected component of $Z$. 

Therefore $A$ and $B$ are in the same connected component of $Z$. 
\end{proof}

\

The following results follows easily from Theorem \ref{main2}. Both results are known but was proved by a case-by-case analysis. 

\begin{cor}
Let $\co$ be an elliptic conjugacy class of $\tW$. Let $\tw, \tw' \in \co_{\min}$. Then $\tw \approx \tw'$. 
\end{cor}
\begin{rmk}
This result was first proved by Geck and Pfeiffer in \cite[3.2.7]{GP00} for $W$ and then by Geck-Kim-Pfeiffer \cite{GKP} for twisted conjugacy classes in exceptional groups and by the first author \cite{He1} in the remaining cases. 
\end{rmk}

\begin{cor}
Let $\co$ be an elliptic conjugacy class of $\tW$ and $\tw \in \co_{\min}$. Then $\t_{\tw}: \cp_{\tw, \tw} \to Z_W(\tw)$ is surjective. 
\end{cor}
\begin{rmk}
This was first conjectured by Lusztig in \cite[1.2]{L2}. He also proved the case where $W$ is of classical type and $\d$ is trivial in \cite{L2}. The twisted conjugacy classes in a classical group were proved by him later (unpublished). The verification of exceptional groups was due to J. Michel. 
\end{rmk}

\section{Good elements}
\subsection{} Let $B^+$ be the braid monoid  associated with $(W,S)$. Then there is a canonical injection $j:W\longrightarrow B^+$ identifying the generators of $W$ with the generators of $B^+$ and $j(w_1w_2)=j(w_1)j(w_2)$ for $w_1,w_2\in W$ if $\ell(w_1w_2)=\ell(w_1)+\ell(w_2)$.

Now the automorphism $\d$ induces an automorphism of $B^+$, which is still denoted by $\d$. Set $\tB^+=\<\d\> \ltimes B^+$. Then $j$ extends in a canonical way to an injection $\tW \to \tB^+$, which we still denote by $j$. We will simply write $\ud{\tw}$ for $j(\tw)$.

Following \cite{G}, we call $\tw\in\tW$ a {\it good} element if there exists a strictly decreasing sequence $S_0\supset S_1\supset\cdots\supset S_l$ of subsets of $S$ and even positive integers $d_0,\cdots,d_l$ such that $$(\u{\tw})^d=\ud{w_0}^{d_0}\cdots\u{w_l}^{d_l}.$$ Here $d$ is the order of $\tw$ and $w_i$ is the maximal element of the parabolic subgroup of $W$ generated by $S_i$. \\
Moreover, if $d$ is even, we call $\tw\in\tW$ {\it very good} if $$(\u{\tw})^{\frac {d}{2}}=\g\ud{w_0}^{\frac {d_0}{2}}\cdots\u{w_l}^{\frac {d_l}{2}}$$ for some $\g\in\<\d\>$.

\subsection{} Let $\tw \in \tW$. Let $\underline \th=(\th_1, \th_2, \cdots, \th_r)$ be a sequence of elements in $\G_{\tw}$ with $\th_1<\th_2<\cdots<\th_r$. We set $F_i=\sum_{j=1}^i V^{\th_j}_{\tw}$ for $0 \le i \le r$. We say that $\underline \th$ is {\it admissible} if $F_r$ contains a regular point of $V$. Then we have a filtration $$0=F_0 \subset \cdots \subset F_r \subset V.$$ Set $W_i=W_{F_i}$. Then $W=W_0 \supset W_1 \supset \cdots \supset W_r=\{1\}$. There there exists  $0=i_0<i_1<i_2< \cdots<i_k \le r$ such that for $0 \le j<k$, $W_{i_j}=W_{i_j+1}=\cdots=W_{i_{j+1}-1} \neq W_{i_{j+1}}$. We then write $r(\underline \th)=(\th_{i_1}, \th_{i_2}, \cdots, \th_{i_k})$ and call it the {\it irredundant  sequence associated to $\underline \th$}. 

For $0 \leq i \leq r$, let $C_i$ be the connected component of $V-\cup_{H\in\fH_{F_i}}H$ containing $A$. We say that a Weyl chamber $A \subset V$ is {\it in good position} with respect to $(\tw, \underline \th)$ if for any $i$, $\bar C_i$ contains some regular point of $F_{i+1}$. It is easy to see that such $A$ always exists. Moreover, $A$ is in good position with respect to $(\tw, \underline \th)$ if and only if the fundamental chamber $C$ is in good position with respect to $(\tw_A, \underline \th)$. 

Let $\underline \th_0$ be the sequence consisting of all the elements in $\G_{\tw}$. We say that a Weyl chamber $A \subset V$ is {\it in good position} with respect to $\tw$ if it is in good position with respect to $(\tw, \underline \th_0)$. 

\begin{lem}\label{bb}
Let $\tw \in \tW$ and $0 \le \th \le \pi$. If $C$ and $\tw(C)$ are in the same connected component of $V-\cup_{H \in \fH_{V_{\tw}^\th}} H$ and $C$ contains a regular point of $V_{\tw}^\th$, then for any $d \in \NN$ with $d \th/2 \pi \in \NN$, we have that $$\underline \tw^d=\s (\underline{w_1 w_0} \, \underline{w_0 w_1})^{d \th/2\pi}.$$ Here $w_1$ is the maximal element in $W_{V_{\tw}^\th}$ and $\s \in \<\d\>$ with $\s(w_1)=w_1$. 

If moreover, $d$ is even and $d \th/2 \pi$ is an odd number, then $$\underline \tw^{d/2}=\s' \underline{w_0 w_1} (\underline{w_1 w_0} \, \underline{w_0 w_1})^{(\frac{d \th}{2 \pi}-1)/2}.$$ Here $\s' \in \<\d\>$ with $\s'(w_0 w_1)=w_1 w_0$. 
\end{lem}
\begin{proof}
We simply write $K$ for $V_{\tw}^\th$ and $J$ for $I(K, C)$. Assume that $\tw \in \t W$ for $\t \in \<\d\>$. Let $v \in C$ be a regular point of $K$. Assume $\th=\frac{2p}{q}\pi$ with integers $p, q$ coprime and $0 \le 2p\le q$. Choose $s,t\in\ZZ$ such that $sp-1=tq$. Then $\tw^{sp}=\tw^{tq} \tw$. Since $\tw^{q}(v)=v$ and $\tw^q$ fixes the connected component of $V-\cup_{H\in\fH_K}H$ containing $C$, we have that $\tw^q(C)=C$. Therefore $\tw^q=\t^q$. Moreover $\t^q(w_1)=w_1$. 

Set $x=\tw^s$. Then $x$ acts on $K$ by rotating $\frac{2 \pi}{q}$ and $K \subset V_x^{\frac{2 \pi}{q}}$. Also $x(K)=K$. Now by Lemma \ref{f1}, $\ell(x^k)=\frac{2 k}{q} \sharp(\fH-\fH_K)$ for any $k \in \NN$ with $2 k \le q$. 

If $2 \mid q$, then $2 \nmid p$ and $\ell(x^{q/2})=\sharp(\fH-\fH_K)$. Also $x \in {}^J \tW^J$ with $x(J)=J$. Thus $x^{q/2}=w_1 w_0 \t^{sq/2}=\t^{sq/2} w_0 w_1$. Hence $\t^{sq/2}(w_0 w_1)=w_1 w_0$. Notice that $\tw=x^p \t^{-tq}$ and $\t^{-tq}(x)=x$. Therefore \begin{align*} \underline \tw^{q/2} &=(\t^{-t q} \underline{x^p})^{q/2}=\t^{-t q^2/2} (\underline x^{q/2})^p=\t^{-t q^2/2} \t^{s p q/2} \underline{w_0 w_1} (\underline{w_1 w_0} \, \underline{w_0 w_1})^{(p-1)/2} \\ &=\t^{q/2} \underline{w_0 w_1} (\underline{w_1 w_0} \, \underline{w_0 w_1})^{(p-1)/2}.\end{align*} Since $\t^q(w_1)=w_1$ and $\t^{sq/2}(w_0 w_1)=w_1 w_0$, we have that $\t^{q/2}(w_0 w_1)=w_1 w_0$ and $\underline \tw^q=\t^q (\underline{w_1 w_0} \, \underline{w_0 w_1})^p$. 

If $2 \nmid q$, then we set $k=\frac{q-1}{2}$.  Then $x^k \in {}^J \tW^J$ and $\ell(x^k)=\frac{q-1}{q} \sharp(\fH-\fH_K)=\ell(w_0 w_1)-\frac{1}{q}\sharp(\fH-\fH_K)$. We have that $\t^{-sk} x^k \in W^J$ and $\t^{-sk} x^k w_1=y \i w_0$ for some $y \in W$ with $\ell(y \i w_0 w_1)=\ell(w_0 w_1)-\ell(y)$ and $\ell(y)=\frac{1}{q} \sharp(\fH-\fH_K)$. 

Similarly, $x^k=w_1 w_0 (y') \i \t^{s k}$ for some $y' \in W$ with $\ell(w_1 w_0 (y') \i)=\ell(w_1 w_0)-\ell(y')$ and $\ell(y')=\frac{1}{q} \sharp(\fH-\fH_K)$.

Since $x^q=(\tw)^{s q}=\t^{s q}$, we have that 
\begin{align*} x &=x^q x^{-k} x^{-k}=\t^{s q} (w_1 w_0 (y') \i \t^{s k}) \i (\t^{s k} y \i w_0 w_1) \i \\ &=\t^{s (q-k)} y' y \t^{-s k}.\end{align*} 

Since $\ell(x)=\frac{2}{q} \sharp(\fH-\fH_K)=\ell(y)+\ell(y')$, we have that $$\underline x=\t^{s (q-k)} \underline{y'} \, \underline{y} \t^{-s k}.$$ 

Moreover, \begin{align*} \t^{s q} &=x^k x x^k=x^k (\t^{s (q-k)} y' y \t^{-s k}) (\t^{s k} y \i w_0 w_1)=x^k \t^{s (q-k)} y' w_0 w_1 \\ &=x^k \t^{s q} (x^{-k}).\end{align*} Hence \begin{align*} \underline x^q &=\underline x^k \underline x \, \underline x^k=\underline x^k \t^{s q} \t^{-s k} \underline{y'} \, \underline y \t^{-s k} \underline x^k= \t^{s q} \underline x^k \t^{-s k} \underline{y'} \, \underline y \t^{-s k} \underline x^k \\ &= \t^{s q} \underline{w_1 w_0 (y') \i} \, \underline{y'} \, \underline y \t^{-s k} \t^{s k} \underline{y \i w_0 w_1}=\t^{s q} \underline{w_1 w_0} \,  \underline{w_0 w_1}. \end{align*}

Thus $$\underline \tw^q=(\t^{-t q} \underline x^p)^q=\t^{-t q^2} (\underline x^q)^p=\t^{-t q^2} \t^{s p q} (\underline{w_1 w_0} \, \underline{w_0 w_1})^p=\t^q (\underline{w_1 w_0} \, \underline{w_0 w_1})^p.$$
\end{proof}

\

Now we prove the existence of good and very good elements. 

\begin{thm}\label{cc}
Let $\tw \in \tW$ and $\underline \th$ be an admissible sequence with $r(\underline \th)=(\th_1, \cdots, \th_k)$. If the fundamental chamber $C$ is in good position with respect to $(\tw, \underline \th)$, then $$\underline \tw^d=\s \underline w_0^{d \th_1/\pi} \underline w_1^{d (\th_2-\th_1)/\pi} \cdots \underline w_{k-1}^{d(\th_k-\th_{k-1})/\pi},$$ here $d \in \NN$ with $d \th_j/2 \pi \in \ZZ$ for all $j$, $w_j$ is the maximal element in $W_{i_j}$ and $\s \in \<\d\>$. 

If moreover, $d$ is even, then $$\underline \tw^{d/2}=\s' \underline w_0^{d \th_1/2 \pi} \underline w_1^{d (\th_2-\th_1)/2 \pi} \cdots \underline w_{k-1}^{d(\th_k-\th_{k-1})/2 \pi}$$ for some $\s' \in \<\d\>$. 
\end{thm}
\begin{proof}
We argue by induction on $\sharp W$. Assume that the statement holds for any $(W', S', \d')$ with $\sharp W'<\sharp W$. 

We assume that $\th$ is irredundant by replacing $\th$ by $r(\th)$ if necessary. By assumption, $\bar C$ contains a regular point of $F_1$. Hence by Proposition \ref{f2}, $\tw=\tw' u$, where $u \in W_{F_1}$, $\tw' \in {}^{I(F_1, C)} \tW^{I(F_1, C)}$ with $\tw'(I(F_1, C))=I(F_1, C)$. 

Set $V_1=F_1^\bot$, $W_1=W_{F_1}$ and $\tW_1=\<\d_1\> \ltimes W_1$, where $\d_1$ is the automorphism on $W_1$ defined by the conjugation of $\tw'$. Then we may naturally regard $\tW_1$ as a reflection subgroup of $GL(V_1)$. Set $C'=C_1 \cap V_1$. Then $C' \subset V_1$ is the fundamental Weyl chamber of $W_1$. Since $C$ is in good position with respect to $\tw$, $C'$ is in good position with respect to $\d_1 u\in \tW_1$. 

By induction hypothesis on $\tW_1$, $$(\d_1 \underline u)^d=(\d_1)^d \underline w_1^{d (\th_2)/\pi} \cdots \underline w_{k-1}^{d(\th_k-\th_{k-1})/\pi}$$ in $\<\d_1\>\rtimes B_1^+$, here $B_1^+$ is the Braid monoid associated with $W_1$. By Lemma \ref{bb}, \begin{align*} \ud \tw^d &=(\ud \tw')^d \underline w_1^{d \th_2/\pi} \cdots \underline w_{k-1}^{d(\th_k-\th_{k-1})/\pi} \\ &=\s (\underline{w_1 w_0} \, \ud{w_0 w_1})^{d\th_1/\pi} \underline w_1^{d \th_2/\pi} \cdots \underline w_{k-1}^{d(\th_k-\th_{k-1})/\pi}.\end{align*}

Since $\underline w_0^2$ commutes with $\ud w_1$, we have that $(\underline{w_1 w_0} \, \ud{w_0 w_1}) \ud w_1^2=\ud w_0^2$. Hence $\ud \tw^d=\s \underline w_0^{d \th_1/\pi} \underline w_1^{d (\th_2-\th_1)/\pi} \cdots \underline w_{k-1}^{d(\th_k-\th_{k-1})/\pi}$.

The 	``moreover'' part can be proved in the same way. 
\end{proof}

\subsection{} It was proved in \cite{G}, \cite{GKP} and \cite{He1} that for any conjugacy class of $\tW$, there exists a good minimal length element. Below we give a case-free proof. We'll also see that it provides a practical way to construct good minimal length element. 

\begin{prop}
Let $\tw\in \tW$ and $A$ be a Weyl chamber. If $A$ is in good position with respect to $\tw$, then $\tw_A$ is a good element and is of minimal length in its conjugacy class. 
\end{prop}

\begin{proof}
We argue by induction on $\sharp W$. The statement is obvious if $W$ is trivial. Now assume that the statement holds for any $(W', S', \d')$ with $\sharp W'<\sharp W$. 

The fundamental alcove $C$ is in good position with respect to $\tw_A$. Hence by Theorem \ref{cc}, $\tw_A$ is good. Set $F=V^W+V_{\tw}$. By definition, $\bar C$ contains some regular point of $V_{\tw}$. By $\S$\ref{aa}, $I(F, C)=I(V_{\tw}, C)$. By Proposition \ref{f2}, $\tw_A=\tw' u$, where $u \in W_F$, $\tw' \in {}^{I(F, C)} \tW^{I(F, C)}$ with $\tw'(I(F, C))=I(F, C)$. Set $V_1=F^\bot$, $W_1=W_F$ and $\tW_1=\<\d_1\> \ltimes W_1$, where $\d_1$ is the automorphism on $W_1$ defined by the conjugation of $\tw'$. The fundamental chamber $C_1 \cap V_1$ of $W_1$ is in good position with respect to $\d_1 u \in \tW_1$. 

By induction hypothesis on $\tW_1$, $\d_1 u$ is of minimal length in its conjugacy class in $\tW_1$. Hence $\tw_A=\tw' u$ is of minimal length in its conjugacy class in $\tW$. 
\end{proof}

\subsection{} Let $w_0$ be the maximal element in $W$. Then $\underline w_0^2$ is a central element in $\tB^+$. Now we discuss some good element $\tw$ such that $\ud \tw^d \in \ud w_0^2 B^+$, where $d$ is the order of $\tw$.

We've shown in the above proposition that for any elliptic conjugacy class in $\tW$, there exists a good minimal length element $\tw$ such that $\ud \tw^d \in \ud w_0^2 B^+$, where $d$ is the order of $\tw$. 

Another example is the conjugacy class of $d$-regular element. We call an element $\tw \in \tW$ {\it $d$-regular} if it has a regular $\xi$-eigenvector, here $\xi$ is a root of unity of order $d$. By \cite[4.10]{Sp} and \cite[Proposition 3.11]{BM}, if $\co$ is a conjugacy class of $\tW$ contains $d$-regular elements, then there exists $\tw \in \co$ such that $\underline \tw^d=\underline w_0^2$. 

\subsection{} We call a conjugacy class $\co$ of $\tW$ {\it quasi-elliptic} if for some (or equivalently, any) $\tw \in \co$, $(V^{\tw})^\bot$ contains a regular point of $V$. Here $V^{\tw}$ is the set of points fixed by $\tw$. Then an elliptic conjugacy class is quasi-elliptic. Also a conjugacy class of $d$-regular elements is also quasi-elliptic. 

Now we have that 

\begin{cor}
Let $\co$ be a quasi-elliptic conjugacy class of $\tW$. Then there exists $\tw \in \co$ such that $\tw$ is good and $\ud \tw^d \in \ud w_0^2 B^+$, here $d$ is the order of $\tw$. 
\end{cor}
\begin{proof}
Let $\tw \in \co$ and $\ud \th$ be the sequence consisting of all nonzero elements in $\G_{\tw}$. Since $\co$ is quasi-elliptic, $\ud \th$ is admissible. Let $A$ be a Weyl chamber in good position with respect to $(\tw, \ud \th)$. Then $C$ is in good position with respect to $(\tw_A, \ud \th)$. By Theorem \ref{cc}, $\tw_A$ is good and $\ud \tw_A^d \in \ud w_0^{d \th_1/\pi} B^+$. Here $\th_1$ is the minimal element in $\ud \th$. Since $d \th_1/2 \pi \in \ZZ$ and $\th_1>0$, we have that $d \th_1/\pi \ge 2$. Hence $\ud \tw_A^d \in \ud w_0^2 B^+$. 
\end{proof}

\section*{Acknowledgement} We are grateful to G. Lusztig for helpful discussion on the centralizer of minimal length element and useful comments on this paper. We thank M. Rapoport for useful comments on Deligne-Lusztig varieties and representation of finite groups of Lie types and thank O. Dudas for helpful discussion on $d$-regular elements.

\end{document}